\documentclass[12pt, reqno]{amsart}
\usepackage{amsmath, amsthm, amscd, amsfonts, amssymb, graphicx, color}
\usepackage[bookmarksnumbered, colorlinks, plainpages]{hyperref}
\hypersetup{colorlinks=true,linkcolor=red, anchorcolor=green, citecolor=cyan, urlcolor=red, filecolor=magenta, pdftoolbar=true}

\newtheorem{theorem}{Theorem}[section]
\newtheorem{lemma}[theorem]{Lemma}

\newtheorem{corollary}[theorem]{Corollary}
\theoremstyle{definition}
\newtheorem{definition}[theorem]{Definition}
\newtheorem{example}[theorem]{Example}

\theoremstyle{remark}

\numberwithin{equation}{section}

\begin{document}
	
	\setcounter{page}{1}
	
	\title[Some Fixed Point Results for Interpolative Contractive Mappings]{Some Fixed Point Results for Interpolative Contractive Mappings in $C^{\ast}$-algebra valued metric spaces }
	
	\author[H. Massit, M. Rossafi, C.  Park]{Hafida Massit$^{1}$, Mohamed Rossafi$^{2}$  \MakeLowercase{and} Choonkil Park$^{3*}$}
	
		\address{$^{1}$Laboratory Partial Differential Equations, Spectral Algebra and Geometry, Department of Mathematics, Faculty of Sciences, University of Ibn Tofail, Kenitra, Morocco}
	\email{massithafida@yahoo.fr}
	
		\address{$^{2}$Laboratory Partial Differential Equations, Spectral Algebra and Geometry, Higher School of Education and Training, University of Ibn Tofail, Kenitra, Morocco}
	\email{mohamed.rossafi1@uit.ac.ma}
	
	\address{$^{3}$Research Institute for Natural Sciences, Hanyang University, Seoul 04763, Republic of Korea}
	\email{\textcolor[rgb]{0.00,0.00,0.84}{baak@hanyang.ac.kr}}

	
	
	\date{$^{*}$Corresponding author}

	\begin{abstract} In this  paper, we present a new concept of interpolative 
		 contraction mappings in $C^{\ast}$-algebra valued complete metric space and we prove the existence of fixed points and common fixed points for Kannan-Riech type contractions.
	\end{abstract}
	
		\maketitle

\baselineskip=12.4pt

	\section{Introduction and preliminaries}

	In $1968$ Kannan introduced a new result that there is a discontinuous mapping that satisfies certain conditions and admits a fixed point in a complete metric spaces.  This result stills one of the most important results that generalizes the famous Banach contraction principle.
	\begin{definition} Let $\mathcal{U}$ be a metric space space. A self-mapping $T:\mathcal{U}\rightarrow\mathcal{U}$ is said to be a Kannan contraction if there exists $\tau \in [0,\dfrac{1}{2}[$ such that 
		\begin{equation*}
			d(Tx,Ty)\leq \tau( d(x,Tx)+d(y,Ty)); \forall x,y\in \mathcal{U}
		\end{equation*}
	\end{definition}
	In 1969, Kannan \cite{KAN} proved that if $  (\mathcal{U},d)$ is complete metric space, then every Kannan contraction on $ \mathcal{U} $ has a unique fixed point. Let $ \mathcal{U} $ be a metric space. A self-mapping $ T $ on $ \mathcal{U} $ is called Reich contraction if there exists  $r\in \left[0,\frac{1}{3} \right[   $ such that 
	$$  d(Tu,Tv) \leq r d(u,v)+ r d(u,Tu) + r d(v,Tv),  \forall  u, v \in \mathcal{U}.$$
	 \\  
	In 1972, Reich \cite{RE} proved that if $  (\mathcal{U},d)$ is complete metric space, then every Reich contraction on $ \mathcal{U} $ has a unique fixed point. It is interesting that Riech’s theorem is generalized of the Banach contraction principle \cite{BA} and Kannan contraction mapping.
	
	Let $ \mathcal{U} $ be a metric space. A self-mapping $ T $ on $ \mathcal{U} $ is called weak contraction  if there exists $ \varphi : \mathcal{U}\rightarrow \mathbb{R^{+}}$ be a lower semicontinuous function and  $ \phi : \mathcal{U}\rightarrow \mathbb{R^{+}}$ be an altering distance function such that 
	$$  \phi(d(Tu,Tv)) \leq  \phi (d(u,v))-\varphi(d(u,v)). $$
	Weak contraction principle was first introduced by Alber et al in Hilbert spaces \cite{YA}.
	
	Recently, Errai et al \cite{YS} announced the notion of interpolative
	Kannan Contraction on $ b $-metric space, they also proved a
	some fixed point results of Interpolative Kannan contraction theorem. Moreover, the notion of Interpolative Kannan contraction are studied by many authors in the field of fixed point theory (see  \cite{M, KARAP} and references therein).
	
\begin{theorem}
If $(\mathcal{U},d)$ is a complete metric space, then
	every Kannan contraction on $\mathcal{U}$ has a unique fixed point.
\end{theorem}	
	In 2018, Karapinar published a new type of contraction
	obtained from the definition of the Kannan contraction by
	interpolation as follows.

	\begin{definition}
		Let $(\mathcal{U},d)$ be a metric space. A self-mapping $T:\mathcal{U}\rightarrow \mathcal{U}$ is said to be an interpolative Kannan type contraction if there are two constants $\tau,\beta \in]0,1[$such that 
		$$d(Tx,Ty)\leq \tau (d(x,Tx)^{\beta})( d(y,Ty))^{1-\beta}$$
		for all $x,y \in \mathcal{U}$ with $x\not=Tx$ and $y\not=Ty$.
	\end{definition}
	Karapinar obtained the following result.
	\begin{theorem}
	Let $(\mathcal{U},d)$	be a complete metric space and $T:\mathcal{U}\rightarrow \mathcal{U}$ be an interpolative Kannan-type contraction mapping. Then $T$  has a unique fixed point.
	\end{theorem}
 Recently, Ma et al \cite{Z.Ma} announced the notion of $ C^{\ast}- $algebra-valued metric space and formulated some first fixed point theorems in the $ C^{\ast}- $algebra-valued metric space. Many authors initiated and studied many existing fixed point theorems in such spaces, see \cite{MR}, \cite{MU}, \cite {S}.\\
 Throughout this paper, we denote $  \mathbb{A}$ by an unital (i.e ,unity element I) $C^{\ast}$-algebra with linear involution $ * $, such that for all $ x,y \in \mathbb{A}  $,
 \begin{center}
 	$ (xy)^{\ast} = y^{\ast}x^{\ast}$, and $ x^{\ast\ast} = x $.
 \end{center}
 
 We call an element $x\in\mathbb{A}$ a positive element, denote it by $x\succeq\theta$
 
 if $x\in\mathbb{A}_{h}=\lbrace x\in \mathbb{A}: x= x^{\ast}\rbrace $ and  $\sigma(x) \subset \mathbb{R}_{+}$,where  $\sigma(x)$  is the spectrum of $ x $. Using positive element, we can define a partial ordering $\preceq$ on $\mathbb{A}_{h}$ as follows : 
 \begin{center}
 	$x\preceq y$  if and only if  $y-x\succeq\theta$
 \end{center}  where $\theta$ means the zero element in $\mathbb{A}$.
 
 \begin{definition}\label{d1} \cite{S} Let $\mathcal{U}$ be a non-empty set 
 	Suppose the mapping
 	$\ d:\mathcal{U}\times \mathcal{U}\rightarrow \mathbb{A}_{+}$ satisfies:
 	\begin{itemize}
 		\item[(i)]  $d(x,y)= 0$ if and only if $x=y$.
 		\item[(ii)]  $d(x,y)=d(y,x)$  for all distinct points $x,y\in \mathcal{U}$. 
 		\item[(iii)]  $d(x,y)\preceq d(x,u)+d(u,y)$  for all $ x, y,u \in \mathcal{U} $. 
 	\end{itemize}
 	Then $\left(\mathcal{U}, \mathbb{A},d\right)$ is called a $ C^{\ast}- $algebra-valued metric space.
 \end{definition}
\begin{lemma} \label{l1} Let $\{x_{n}\}$ be a sequence in a $ C^{\ast}- $algebra-valued metric space $\left(\mathcal{U}, \mathbb{A},d\right)$ such that 
	\begin{equation}
		d(x_{n},x_{n+1}) \preceq \delta d(x_{n-1},x_{n})
	\end{equation}
	for some $\delta \in[0,1[$ and $\forall n\in \mathbb{N}$. \\Then, $\{x_{n}\}$ is a Cauchy sequence in $\left(\mathcal{U}, \mathbb{A},d\right)$.
\end{lemma}
\begin{proof} First we denote that 
	\begin{equation*}
		d(x_{n+1},x_{n})\preceq \delta^{n} d(x_{1},x_{0}), \; \forall n\in \mathbb{N}.
	\end{equation*}
	For $ m\geq 1 $ and $ p\geq 1 $
	\begin{align*}
		d( x_{m},x_{m+p})&\preceq b\left(  d( x_{m},x_{m+1})+d( x_{m+1},x_{m+p})\right)\\
		&\preceq   d( x_{m},x_{m+1})+d( x_{m+1},x_{m+2}+...+ d( x_{m+p-2},x_{m+p-1})\\&+ d( x_{m+p-1},x_{m+p})\\
		&\preceq \delta ^{m}d( x_{0},x_{1})+\delta ^{m+1}d( x_{0},x_{1})+\delta ^{m+2} d( x_{0},x_{1})+\delta ^{m+3}d( x_{0},x_{1})\\
		&+......+\delta ^{m+p-1}d( x_{0},x_{1})\\
		&=\delta^{m}+\delta^{m+1}+...+\delta^{m+p-1})d(x_{0},x_{1})\\
		&= \sum_{k=0}^{p-1}\delta^{m+k} d(x_{0},x_{1}).
	\end{align*}
	Since  $ \delta\in [0,1) $ ,  \\
	$\Vert d( x_{m},x_{m+p})\Vert \leq  \sum_{k=0}^{p-1}(\delta)^{m+k} \Vert d(x_{0},x_{1})\Vert \rightarrow 0 \;\text{quand}\; m\rightarrow \infty$.
	
	Then,
	\begin{equation*}
		\displaystyle\lim_{n,m\rightarrow \infty }d({x_n},{x_m})=\theta.
	\end{equation*}
	This implies that $ \{x_{n} \}_{n}$  is a Cauchy sequence.
	
\end{proof}
\section{Main result}
	inspired by the notion of interpolative Kannan in complete metric space, we present a new concepts of interpolative Kannan-Reich weak type contraction mappings.
we discuss the existence of fixed point results for these new contractive mappings in  $C^{\ast}$-algebra valued complete metric space.
\begin{definition}
	Let $(\mathcal{U},\mathcal{A},d)$ be a $ C^{\ast}- $algebra-valued metric space. A self-mapping $T:\mathcal{U}\rightarrow \mathcal{U}$ is called an interpolative Kannan type contraction if there exist $\tau, \beta \in]0,1[$ such that  
	\begin{equation}
		d(Tx,Ty)\preceq \tau d(x,Tx)^{\beta}d(y,Ty)^{1-\beta}; \forall x,y\in \mathcal{U}
	\end{equation} with $Tx\not=x$.
\end{definition}
\begin{theorem}\label{t1}
	Let $(\mathcal{U},\mathcal{A},d)$ be a complete $ C^{\ast}- $algebra-valued metric space and $T:\mathcal{U}\rightarrow \mathcal{U}$ an interpolative Kannan type contraction.\\  Then $T$ has a fixed point in $\mathcal{U}$.
\end{theorem}
\begin{proof} Let $ x_{0} \in \mathcal{U} $ and define a sequence $ \{x_{n}\} \in \mathcal{U} $ by $x_{0}=x$ and $ x_{n+1}= Tx_{n} $, $ \forall n \in \mathbb{N} $. Suppose that $ x_{n}\not= Tx_{n} $, $ \forall n \in \mathbb{N} $ .\\
We have 

\begin{align*} d(x_{n+1},x_{n})=d(Tx_{n},Tx_{n-1}) &\preceq \tau d(x_{n},Tx_{n})^{\beta}d(x_{n-1},Tx_{n-1})^{1-\beta}\\
&= \tau d(x_{n-1},x_{n})^{1-\beta}d(x_{n},x_{n+1})^{\beta}\\
\end{align*}
this implies that \begin{equation}\label{e1}
	d(x_{n},x_{n+1})^{1-\beta}
	\preceq \tau d(x_{n-1},x_{n})^{1-\beta}
\end{equation}
Thus, we have that the sequence $\{d(x_{n}, x_{n-1})\}$ is non-increasing and non-negative. 
From \ref{e1},
we deduce that \begin{equation}\label{e2}
		d(x_{n},x_{n+1})
	\preceq \tau d(x_{n-1},x_{n})\preceq \tau^{n}d(x_{0},x_{1}).
\end{equation} 
Letting $n\rightarrow \infty$ in the inequality \ref{e2}, we obtain $\displaystyle\lim_{n\rightarrow\infty} d(x_{n},x_{n+1})=\theta$. \\
	For $ m \geq 1 $ and $ r \geq 1 $, it follows that 
\begin{align*}
	d(x_{m},x_{m+r}) & \preceq  d(x_{m},x_{m+1})+ d(x_{m+1},x_{m+2})
		+...+d(x_{m+r-1},x_{m+r})\\
	&\preceq  \tau^{m} d(x_{0},x_{1})+ ...+ \tau^{m+r-1}d(x_{0},x_{1})\\
	&\preceq \dfrac{\tau^{m}}{1-\tau}  d(x_{0},x_{1}) \rightarrow \theta \;\;as \;\;m\rightarrow \infty. \\
\end{align*}
Hence $ \{ x_{n}\} $ is a Cauchy sequence in $( \mathcal{U} ,\mathcal{A} ,d) $. Hence there exist $ z\in \mathcal{U} $ such that 
\begin{center}
	$\displaystyle\lim_{n\rightarrow \infty}  d(x_{n},z)= \theta $. 
\end{center}
Now, we shall show that $ z $ is fixed point of $ T $.
we get 
\begin{equation*}
d(Tx_{n},Tz) \preceq \tau d(x_{n}, Tx_{n})^{\beta} d(z,Tz)^{1-\beta}
\end{equation*}
Letting $ n \rightarrow \infty $ and using the concept of continuity of the function of $ T $.
\\We have $ d(z,Tz) = \theta$.\\ For uniqueness, Let $ y $ be another fixed point of $ T $  where:
 \begin{align*}
 d(z,y)&=	d(Tz,Ty)\\
 &\preceq \tau d(z,Tz)^{\beta}d(y,Ty)^{1-\beta}\\
 	&=0
 \end{align*}
 then $z=y$.
\end{proof}
\begin{example}
	Let $X=]2, \infty[$, $\mathcal{A}= \mathbb{R}^{2} $ and $d$ the metric defined on $X$ by $d: X\times X \rightarrow \mathbb{R}^{2}$ $d(x,y)= ((x+y)^{2}, 0)$.\\ Define a self-mapping $T$ by $Tx= \dfrac{1}{x}$. $T$ is a interpolative Kannan type contraction for $\tau=\dfrac{3}{4} $ and $\alpha= \dfrac{2}{5}$.
\end{example}
\begin{definition}
	Let $(\mathcal{U},\mathcal{A},d)$ be a $ C^{\ast}- $algebra-valued metric space. A self-mapping $T:\mathcal{U}\rightarrow \mathcal{U}$ is called $(\tau, \beta,\eta)$-interpolative Kannan contraction if $\tau, \beta,\;\eta \in]0,1[$ such that $\beta+\eta < 1$ and 
		\begin{equation}
		d(Tx,Ty)\preceq \tau d(x,Tx)^{\beta}d(y,Ty)^{\eta}; \forall x,y\in \mathcal{U}
	\end{equation} with $Tx\not=x$, $Ty\not=y$.
\end{definition}
\begin{theorem}
	Let $(\mathcal{U},\mathcal{A},d)$ be a complete $ C^{\ast}- $algebra-valued metric space and $T:\mathcal{U}\rightarrow \mathcal{U}$ be a $(\tau, \beta,\eta)$-interpolative Kannan contraction. 
 Then $T$ has a fixed point in $\mathcal{U}$.
\end{theorem}
\begin{proof}Let $ x_{0} \in \mathcal{U} $ and define a sequence $ \{x_{n}\} \in \mathcal{U} $ by $x_{0}=x$ and $ x_{n+1}= Tx_{n} $, $ \forall n \in \mathbb{N} $. Suppose that $ x_{n}\not= Tx_{n} $, $ \forall n \in \mathbb{N} $ .\\
	We have 
	$$ d(x_{n},x_{n+1})=d(Tx_{n-1},Tx_{n}) \preceq \tau d(x_{n-1},x_{n})^{\beta}d(x_{n},x_{n+1})^{\eta}$$
	ie.
	$$d(x_{n},x_{n+1})^{1-\eta}\preceq \tau d(x_{n-1},x_{n})^{\beta}\preceq \tau d(x_{n-1},x_{n})^{1-\eta}$$
	since $\beta\leq 1-\eta$.\\ Similary to the proof of the Theorem \ref{t1} we obtain the result.
	
\end{proof}

\begin{definition}	Let $(\mathcal{U},\mathcal{A},d)$ be a $ C^{\ast}- $algebra-valued metric space and  $T, S:\mathcal{U}\rightarrow \mathcal{U}$ be  two self-maps. We call $(T,S)$ a $(\tau, \beta,\eta)$-interpolative Kannan contraction pair, if there exist $\tau\in[0,1[,\; 0<  \beta, \eta< 1$ with $\beta+\eta<1$ such that 
	\begin{equation}
		d(Tx,Sy) \preceq \tau d(x,Tx)^{\beta} d(y,Sy)^{\eta}
	\end{equation}
	for all $x,y\in \mathcal{U}$ with $x\not= Tx$, $y\not= Sy$.
\end{definition}
\begin{theorem}	Let $(\mathcal{U},\mathcal{A},d)$ be a complete $ C^{\ast}- $algebra-valued metric space and $(T,S)$ be a $(\tau, \beta,\eta)$-interpolative Kannan contraction pair. Then $T$ and $S$ have a unique common fixed point in $\mathcal{U}$. i.e., there exists $z \in\mathcal{U}$ such that $Tz=z= Sz$.
	
\end{theorem}
\begin{proof} Let $ x_{0} \in \mathcal{U} $ and  we define a sequence $ \{x_{n}\} \in \mathcal{U} $ by\\
	 $Tx_{2n+1}=x_{2n+2}$ and $ Sx_{2n}= x_{2n+1} $, $ \forall n \in \mathbb{N} $.
\begin{align*}
	d(x_{2n+1},x_{2n+2}) &\preceq \tau d(x_{2n},x_{2n+1})^{\beta}d(x_{2n},x_{2n+1})^{\eta}\\
	&\preceq \tau d(x_{2n},x_{2n+1})^{\beta}d(x_{2n+1},x_{2n+2})^{1-\beta}	
\end{align*}
Thus, 
\begin{equation*}
	d(x_{2n+1},x_{2n+2})^{\beta}\preceq \tau d(x_{2n},x_{2n+1})^{\beta}
\end{equation*}
	then, \begin{align*}
		d(x_{2n+1},x_{2n+2})&\preceq \tau^{\frac{1}{1-\beta}} d(x_{2n},x_{2n+1})\\
		&\preceq \tau d(x_{2n},x_{2n+1}).
	\end{align*}
	Hence, 
	\begin{align*}
		d(x_{2n+1},x_{2n+2}) 
		&\preceq \tau d(x_{2n},x_{2n+1})\\
		&\preceq \tau^{2} d(x_{2n-1},x_{2n})\\
		&\preceq .\\
		&\preceq.\\
		&\preceq \tau^{2n+1}d(x_{0},x_{1}).
	\end{align*}
	Similarly, \begin{align*}
		d(x_{2n+1},x_{2n})&= d(Tx_{2n}, Sx_{2n-1})\\
		&\preceq \tau d(x_{2n},x_{2n+1})^{\beta}d(x_{2n-1},x_{2n})^{1-\beta}\\
		\end{align*}
	we have
	\begin{equation*}
		 d(x_{2n+1},x_{2n})^{1-\beta} \preceq \tau d(x_{2n-1},x_{2n})^{1-\beta}
	\end{equation*}
	then, \begin{align*}
		 d(x_{2n+1},x_{2n}) 	&\preceq \tau^{\frac{1}{1-\beta}} d(x_{2n-1},x_{2n})\\
		 	&\preceq \tau d(x_{2n-1},x_{2n}).
	\end{align*}
	Hence, 	 \begin{align}\label{e2}
		d(x_{2n+1},x_{2n+2})&\preceq \tau^{\frac{1}{1-\beta}} d(x_{2n},x_{2n+1})\\
		&\preceq \tau d(x_{2n},x_{2n+1}).
	\end{align}
	And
	\begin{align*}
		d(x_{2n+1},x_{2n}) 
		&\preceq \tau d(x_{2n-1},x_{2n})\\
		&\preceq \tau^{2} d(x_{2n-2},x_{2n-1})\\
		&\preceq .\\
		&\preceq.\\
		&\preceq \tau^{2n}d(x_{0},x_{1}).
	\end{align*}
So \begin{equation}\label{e3}
	d(x_{2n+1},x_{2n})\preceq \tau^{2n} d(x_{0},x_{1}).
\end{equation}
From \ref{e2} and \ref{e3}, we deduce that 
\begin{equation}\label{e4}
	d(x_{n},x_{n+1})\preceq \tau^{n} d(x_{0},x_{1}).
\end{equation}
	Using \ref{e4} we prove that $\{x_{n}\}$ is a Cauchy sequence.\\
		For $ m \geq 1 $ and $ r \geq 1 $, it follows that 
	\begin{align*}
		d(x_{m},x_{m+r}) & \preceq  d(x_{m},x_{m+1})+ d(x_{m+1},x_{m+2})
		+...+d(x_{m+r-1},x_{m+r})\\
		&\preceq  \tau^{m} d(x_{0},x_{1})+ ...+ \tau^{m+r-1}d(x_{0},x_{1})\\
		&\preceq \dfrac{\tau^{m}}{1-\tau}  d(x_{0},x_{1}) \rightarrow \theta \;\;as \;\;m\rightarrow \infty.\\
	\end{align*} 
	Hence $ \{ x_{n}\} $ is a Cauchy sequence in $( \mathcal{U} ,\mathcal{A} ,d) $. Hence there exist $ z\in \mathcal{U} $ such that 
	\begin{center}
		$\displaystyle\lim_{n\rightarrow \infty}  d(x_{n},z)= \theta $ 
	\end{center}
	Now, we shall show that $ z $ is fixed point of $ T $.
	we get 
	\begin{align*}
		d(Tz,x_{2n+2}) &= 	d(Tz,Sx_{2n+1})\\&\preceq \tau d(z, Tz)^{\beta} d(x_{2n+1},x_{2n+2})^{1-\beta}
	\end{align*}
	Letting $ n \rightarrow \infty $ and using the continuity of the function of $ T $.
	\\We have $ d(z,Tz) = \theta$. Similarly,
	
	\begin{align*}
		d(x_{2n+2},Sz) &= 	d(Tx_{2n},Sz)\\&\preceq \tau  d(x_{2n},x_{2n+1})^{\beta}d(z, Sz)^{1-\beta}.
	\end{align*}
	Letting $n\rightarrow \infty$ we obtain $Sz=z$.
	\\ For uniqueness, Let $ y $ be another common fixed point of $ T $ and $S$, then  
	\begin{align*}
		d(z,y)&=	d(Tz,Sy)\\
		&\preceq \tau d(z,Tz)^{\beta}d(y,Sy)^{1-\beta}\\
		&=0
	\end{align*}
	then $z=y$.
\end{proof}

\begin{definition}Let $(\mathcal{U},\mathcal{A},d)$ be a $ C^{\ast}- $algebra-valued metric space and  $T, R:\mathcal{U}\rightarrow \mathcal{U}$ be  two self-maps. We call $T$ is an $R$-interpolative Kannan type contraction , if there exist $\tau\in[0,1[ \:\text{and} \;0<  \beta< 1$  such that 
	\begin{equation}
		d(Tx,Ty) \preceq \tau d(Rx,Tx)^{\beta} d(Ry,Ty)^{1-\beta}
	\end{equation}
$ 	\forall x,y\in \mathcal{U}$ with $x\not= Tx$, $y\not= Ty$.
	
\end{definition}
\begin{theorem}Let $(\mathcal{U},\mathcal{A},d)$ be a complete $ C^{\ast}- $algebra-valued metric space and  $T :\mathcal{U}\rightarrow \mathcal{U}$ is a $R$-interpolative Kannan contraction. Assume that $T\mathcal{U} \subset R\mathcal{U}$ and  $R\mathcal{U}$ is closed. If there exist $\tau\in[0,1[ \;\text{and}\; 0<  \beta< 1$ such that 
	\begin{equation}
	d(Tx,Ty) \preceq \tau d(Rx,Tx)^{\beta} d(Ry,Ty)^{1-\beta}
\end{equation}
$ 	\forall x,y\in \mathcal{U}$ with $x\not= Tx$, $y\not= Ty$.
	Then, $T$ and $R$ have a unique common fixed point in $\mathcal{U}$.
	
\end{theorem}
\begin{proof}Let $ x_{0} \in \mathcal{U} $ and we define a sequence $ \{x_{n}\} \in \mathcal{U} $ by\\
	$Rx_{n+1}=Tx_{n}$ $ \forall n \in \mathbb{N} $.
	\begin{align*}
		d(Tx_{n+1},Tx_{n}) &\preceq \tau d(Rx_{n+1},Tx_{n+1})^{\beta}d(Rx_{n},Tx_{n})^{1-\beta}\\
		&= \tau d(Tx_{n},Tx_{n+1})^{\beta}d(Tx_{n-1},Tx_{n})^{1-\beta}	
	\end{align*}
	Then, \begin{equation*}
		d(Tx_{n+1},Tx_{n})\preceq \tau^{\frac{1}{1-\beta}} d(Tx_{n},Tx_{n-1})\preceq \tau d(Tx_{n},Tx_{n-1}).
	\end{equation*}
	Lemma \ref{l1} implies that $\{Tx_{n}\}$ is a Cauchy sequence and consequently $\{Rx_{n}\}$ is also a Cauchy sequence.
	
	Let $z\in\mathcal{U}$ such that $\displaystyle\lim_{n\rightarrow\infty} d(Tx_{n}, z)=\displaystyle\lim_{n\rightarrow\infty} d(Rx_{n+1}, z)= \theta$. 
	Since $z\in R\mathcal{U}$, there exists $ v\in \mathcal{U} $ such that $z=Rv$.  \\ We obtain 
	\begin{equation*}
		d(Tx_{n},Tz)\preceq \tau d(Rx_{n},Tx_{n})^{\beta} d(Rv,Tv)^{1-\beta}.
	\end{equation*}
	Letting $n\rightarrow \infty$ we obtain $z= Rv= Tv$.
\end{proof}
\begin{example}	Let $X=]2, \infty[$, $\mathcal{A}= \mathbb{R}^{2} $ and $d$ the metric defined on $X$ by $d: X\times X \rightarrow \mathbb{R}^{2}$ $d(x,y)= ((x+y)^{2}, 0)$.\\ Define two self-mapping $T$ and $R$ by $Tx= \dfrac{1}{x}$ and $R(x)= x^{2}$. $T$ is a $R-$interpolative Kannan contraction for $\tau=\dfrac{3}{4} $ and $\alpha= \dfrac{2}{5}$.

\end{example}
 \begin{theorem}Let $(\mathcal{U},\mathcal{A},d)$ be a complete $ C^{\ast}- $algebra-valued metric space and  $T :\mathcal{U}\rightarrow \mathcal{U}$
  satisfy the following condition:
  \begin{equation}\label{e12}
  	d(Tx,Ty) \preceq \tau d(x,y)^{\alpha} d(x,Ty)^{\beta} d(y,Ty)^{\eta}
  \end{equation}
  	$\forall x,y \in\mathcal{U}$ with $x\not= Tx$, $y\not=Ty $, where $\tau \in ]0,1[$ and $\alpha,\beta,\eta \in ]0,1[$ such that $\alpha+\beta+\eta >1$. If there exists $x_{0}\in \mathcal{U}$ such that $d(x_{0},Tx_{0}) \preceq I$ 
  	, then $T$ has a fixed point in $\mathcal{U}$.
 \end{theorem}
 \begin{proof} Let $x_{0}\in \mathcal{U}$ and define a sequence $ \{x_{n}\} \in \mathcal{U} $ by $x_{0}=x$ and $ x_{n+1}= Tx_{n} $, $ \forall n \in \mathbb{N} $. Suppose that $ x_{n}\not= Tx_{n} $, $ \forall n \in \mathbb{N} $ .\\
 	We have 
 	
 	\begin{align*} d(x_{n},x_{n+1})=d(Tx_{n-1},Tx_{n}) &\preceq \tau  d(x_{n-1},x_{n})^{\alpha}d(x_{n-1},Tx_{n-1})^{\beta}d(x_{n},Tx_{n})^{\eta}\\
 		&= \tau d(x_{n-1},x_{n})^{\alpha}d(x_{n-1},x_{n})^{\beta}d(x_{n},x_{n+1})^{\eta}\\
 			&= \tau d(x_{n-1},x_{n})^{\alpha+\beta}d(x_{n},x_{n+1})^{\eta}.
 	\end{align*}
 	Therefore, \begin{align*} d(x_{n},x_{n+1}) &\preceq (\tau  d(x_{n-1},x_{n})^{\alpha+ \beta})^{\frac{1}{1-\eta}}\\
 		&= \tau^{\frac{1}{1-\eta}} d(x_{n-1},x_{n})^{\frac{\alpha+\beta}{1-\eta}}\\
 		&\preceq  \tau^{\frac{2}{1-\eta}} d(x_{n-2},x_{n-1})^{\frac{2(\alpha+\beta)}{1-\eta}}\\
 		&\preceq ...\\
 		&\preceq  \tau^{\frac{n}{1-\eta}} d(x_{0},x_{1})^{\frac{n(\alpha+\beta)}{1-\eta}}.
 		\end{align*}
 		Then, \begin{equation}
 			d(x_{n},x_{n+1})\preceq \tau^{\frac{n}{1-\eta}}.
 		\end{equation}
 	Since $\tau<1$ and $\frac{n}{1-\eta}>1 $, we have
 		\begin{equation}
 			\displaystyle\lim_{n\rightarrow\infty} d(x_{n},x_{n+1})=\theta.
 		\end{equation}
 	For  $ m \geq 1 $ and $ r \geq 1 $, it follows that 
 	\begin{align*}
 		d(x_{m},x_{m+r}) & \preceq  d(x_{m},x_{m+1})+ d(x_{m+1},x_{m+2})
 		+...+d(x_{m+r-1},x_{m+r})\\
 		&\preceq  \tau^{\frac{m}{1-\eta}}+ ...+ \tau^{\frac{m+r-1}{1-\eta}}\\
 		&\preceq \tau^{\frac{m}{1-\eta}}\dfrac{1+\tau}{1-\tau^{2}}   \rightarrow \theta \;\;as \;\;m\rightarrow \infty. \\
 	\end{align*} 
 	Hence $ \{ x_{n}\} $ is a Cauchy sequence in $( \mathcal{U} ,\mathcal{A} ,d) $. Hence there exist $ z\in \mathcal{U} $ such that 
 	\begin{center}
 		$\displaystyle\lim_{n\rightarrow \infty}  d(x_{n},z)= \theta $ 
 	\end{center}
 	Now, we shall show that $ z $ is fixed point of $ T $.
 	we get 
 	\begin{equation*}
 		d(Tx_{n},Tz) \preceq \tau d(x_{n}, z)^{\alpha} d(x_{n}, Tz)^{\beta} d(z,Tz)^{\eta}
 	\end{equation*}
 	Letting $ n \rightarrow \infty $ and using the  continuity of  $ T $.
 	\\We have $ d(z,Tz) = \theta$.
 \end{proof}
\begin{definition}
	A function $\psi: \mathcal{A}^{+} \rightarrow \mathcal{A}^{+}$ is called an alterning distance function  if it satisfies the following conditions:
	\begin{itemize}
		\item [(i)] $\psi$ is continuous,
		\item [(ii)]$\psi$ is nondecreasing,
		\item[(iii)] $\psi(x)= \theta$ if and only if $x=\theta$.
	\end{itemize}
\end{definition}
\begin{example}
Let	 $ \psi  :\mathbb{R}^{2}_{+} \rightarrow \mathbb{R}^{2}_{+}$ a function, for  $ t=(x,y) \in \mathbb{R}^{2}$ ,\\ given by: 
	$$   \psi(t)=  \left\{ \begin{array}{ll}
		(x,y) \: \:\text{if}\; x \leq 1 \:  \:\text{and} \:y \leq 1\\
		(x^{2},y) \: \:  \text{if}\; x >1\;\text{and} \;y \leq 1 \\
		(x,y^{2})   \: \:\text{if} \; x\leq 1 \:\text{and}\: y>1 \\
		(x^{2},y^{2})  \:\:\text{if}\; x>1 \: \text{and} \:y>1.        
	\end{array}\right. $$   
 $\psi$ is an alterning distance function. 
\end{example}
\begin{definition} Let $(\mathcal{U},\mathcal{A},d)$ be a complete $ C^{\ast}- $algebra-valued metric space and  $T :\mathcal{U}\rightarrow \mathcal{U}$. The mapping $T$ is defined as an interpolative weakly contractive mapping of the Reich type if it  
	satisfy the following condition:
	\begin{equation}
		\phi (d(Tx,Ty)) \preceq \phi(d(x,y)^{\alpha} d(x,Ty)^{\beta} d(y,Ty)^{\eta})- \psi(d(x,y)^{\alpha} d(x,Ty)^{\beta} d(y,Ty)^{\eta})
	\end{equation}
	$\forall x,y\in\mathcal{U}$, such that $\phi$ is an alterning distance function and $\psi$ a mapping from $\mathcal{A}_{+}$ to $\mathcal{A}_{+}$ is characterized by its lower semicontinuity satisfying the condition that $\psi(x)= \theta \Leftrightarrow x=\theta$. 
	\end{definition}
	\begin{theorem}\label{t7} Let $(\mathcal{U},\mathcal{A},d)$ be a complete $ C^{\ast}- $algebra-valued metric space and  $T :\mathcal{U}\rightarrow \mathcal{U}$ is an interpolative weakly contractive mapping of the Reich type, $\forall x,y\in\mathcal{U}$ with $x\not= Tx$ and $y\not= Ty$. Let $\alpha, \beta ,\eta \in ]0,1[$ such that $\alpha+\beta+\eta =1$.
		If there exists a point $x_{0}\in\mathcal{U}$ satisfying $d(x_{0},Tx_{0})\preceq I$ then,  $T$ has a fixed point in $\mathcal{U}$. 
		
	\end{theorem}
	\begin{proof}Let $x_{0}\in \mathcal{U}$ and define a sequence $ \{x_{n}\} \in \mathcal{U} $ by $x_{0}=x$ and $ x_{n+1}= Tx_{n} $, $ \forall n \in \mathbb{N} $. Suppose that $ x_{n}\not= Tx_{n} $, $ \forall n \in \mathbb{N} $ .\\
		We have 	\begin{align*} \phi(d(x_{n},x_{n+1})) &\preceq \phi( d(x_{n-1},x_{n})^{\alpha}d(x_{n-1},Tx_{n-1})^{\beta}d(x_{n},Tx_{n})^{\eta})\\
			&-\psi( d(x_{n-1},x_{n})^{\alpha+\beta}d(x_{n},Tx_{n})^{\eta})\\
			&\preceq \phi(d(x_{n-1},x_{n})^{\alpha+\beta}d(x_{n},x_{n+1})^{\eta})
	\end{align*}
		This implies that \begin{align*}
			d(x_{n},x_{n+1})&\preceq d(x_{n-1},x_{n})^{\alpha+\beta}d(x_{n},x_{n+1})^{\eta}\\
			&\preceq d(x_{n-1},x_{n})^{\frac{\alpha+\beta}{1-\eta}}\\
				&\preceq d(x_{n-1},x_{n})\\
			&\preceq d(x_{0},x_{1})\\
			&\preceq I.	
		\end{align*}
		Then, there exists $\gamma \succeq \theta $ such that 
		\begin{equation*}
			\displaystyle\lim_{n\rightarrow \infty} d(x_{n},x_{n+1})=\gamma.
		\end{equation*}
		Letting $n\rightarrow \infty$ in \begin{equation*} \phi(d(x_{n},x_{n+1})) \preceq \phi( d(x_{n-1},x_{n})^{\alpha+\beta}d(x_{n},Tx_{n})^{\eta})-\psi(d(x_{n-1},x_{n})^{\alpha+\beta}d(x_{n},Tx_{n})^{\eta})
			\end{equation*}
			we obtain \begin{equation*}
				\phi(\gamma)\preceq \phi(\gamma^{\alpha+\beta}\gamma^{\eta})-\psi(\gamma^{\alpha+\beta}\gamma^{\eta})\preceq \phi(\gamma)-\psi(\gamma) .
			\end{equation*}
			This implies that 	\begin{equation*}
				\displaystyle\lim_{n\rightarrow \infty} d(x_{n},x_{n+1})=\theta.
			\end{equation*} 
			For $ m \geq 1 $ and $ p \geq 1 $, it follows that 
			\begin{align*}
				\phi(d(x_{m+1},x_{p+1})) &\preceq \phi( d(x_{m},x_{p})^{\alpha}d(x_{m},x_{m+1})^{\beta}d(x_{p},x_{p+1})^{\eta})-\psi(d(x_{m},x_{p})^{\alpha}d(x_{m},x_{m+1})^{\beta}d(x_{p},x_{p+1})^{\eta})\\
			 	&   \rightarrow \theta \;\;as \;\;m,p\rightarrow \infty. \\
			\end{align*} 
			Then, $\{x_{n}\}$ is a Cauchy sequence, and there exists $z\in\mathcal{U}$ such that $\displaystyle\lim_{n\rightarrow \infty} d(x_{n}, z)= \theta$.\\
			We have \begin{equation*}
				d(x_{n},z)^{\alpha}d(x_{n},x_{n+1})^{\beta} d(z,Tz)^{\eta}\rightarrow \theta.
			\end{equation*}
			We get \begin{equation*}
			\phi(d(z,Tz))\preceq\phi(d(x_{n},z)^{\alpha}d(x_{n},x_{n+1})^{\beta} d(z,Tz)^{\eta})-\psi(	d(x_{n},z)^{\alpha}d(x_{n},x_{n+1})^{\beta} d(z,Tz)^{\eta}).
			\end{equation*}
			Hence $d(z,Tz)=\theta$.
		\end{proof}
		\begin{corollary}Let $(\mathcal{U},\mathcal{A},d)$ be a complete $ C^{\ast}- $algebra-valued metric space and  $T :\mathcal{U}\rightarrow \mathcal{U}$ a mapping such that there exists $k\in]0,1[$
			\begin{equation*}
				d(Tx,Ty)\preceq k[d(x,y)^{\alpha}d(x,Tx)^{\beta} d(y,Ty)^{\eta}]
			\end{equation*}
			where $\alpha,\beta,\eta \in]0,1[$ and $\alpha
		+\beta+\eta =1$.\\ If there exists a point $x_{0}\in\mathcal{U}$ satisfying $d(x_{0},Tx_{0})\preceq I$ then,  $T$ has a fixed point in $\mathcal{U}$. 
		\end{corollary}
		\begin{proof}
			It is sufficient to apply $\phi(x)= x$ and $\psi(x)=(I-k)x$ in Theorem \ref{t7}.
		\end{proof}

\medskip

\section*{Declarations}

\medskip

\noindent \textbf{Availablity of data and materials}\newline
\noindent Not applicable.

\medskip

\noindent \textbf{Human and animal rights}\newline
\noindent We would like to mention that this article does not contain any studies
with animals and does not involve any studies over human being.

\medskip

\noindent \textbf{Conflict of interest}\newline
\noindent The authors declare that they have no competing interests.

\medskip

\noindent \textbf{Fundings} \newline
\noindent The authors declare that there is no funding available for this paper.

\medskip

\noindent \textbf{Authors' contributions}\newline
\noindent The authors equally conceived of the study, participated in its
design and coordination, drafted the manuscript, participated in the
sequence alignment, and read and approved the final manuscript. 

\medskip


\end{document}